%% file: RadonStability.tex
\documentclass[a4]{article}
\usepackage[utf8]{inputenc} 
\usepackage{amsmath, amsfonts, amsthm, amssymb}
\usepackage{url}
\usepackage{graphicx}
\usepackage{color}
\usepackage{authblk}

\newcommand{\R}{\mathbf{R}}
\newcommand{\N}{\mathbf{N}}

\newcommand{\T}{\mathbf{T}}
\newcommand{\del}{\partial}
\newcommand{\eps}{\varepsilon}
\newcommand{\bx}{\mathbf{x}}
\newcommand{\by}{\mathbf{y}}
\newcommand{\e}{\text{e}}

\renewcommand{\d}{\,\text{d}}
\renewcommand{\i}{\text{i}}

\DeclareMathOperator{\supp}{supp}

\newtheorem{theorem}{Theorem}[section]
\newtheorem{lemma}[theorem]{Lemma}
\newtheorem{proposition}[theorem]{Proposition}
\newtheorem{corollary}[theorem]{Corollary}
\newtheorem{remark}[theorem]{Remark}
\newtheorem{definition}[theorem]{Definition}

\numberwithin{equation}{section}

\makeatletter
\let\@fnsymbol\@arabic
\makeatother

\title{Stability estimates with a priori bound for the inverse local Radon transform}

\author{Joel Andersson\thanks{Email: \protect\url{joelan@math.su.se}} }
\author{Jan Boman\thanks{Email: \protect\url{jabo@math.su.se}}}
\affil{\textsc{Department of Mathematics, Stockholm University, SE-10691 Stockholm, Sweden}}

\date{\today}

\begin{document}
\maketitle
\begin{abstract}
We consider the inverse problem for the $2$-dimensional weighted local Radon transform $R_m[f]$, where $f$ is supported in $y\geq x^2$ and $R_m[f](\xi,\eta)=\int f(x, \xi x + \eta) m(\xi, \eta, x)\d x$ is defined near $(\xi,\eta)=(0,0)$. For weight functions satisfying a certain differential equation we give weak estimates of $f$ in terms of $R_m[f]$ for functions $f$ that satisfies an a priori bound.
\end{abstract}


\section{Introduction}
\input{Intro.tex}

\section{The means $\mathcal{M}_{\eps,\gamma}$}
\input{Means.tex}

\section{Estimate for the standard Radon transform}
\input{RT.tex}

\section{Estimate for the weighted Radon transform}
\input{WRT.tex}

\noindent\emph{The authors would like to thank Institut Mittag-Leffler for providing an excellent working environment during the time when this research was conducted. We would also like to thank Professor Jan-Olov Str\"omberg and Professor Mikko Salo for valuable discussions regarding this work.}

\section{Appendix}
\input{Appendix.tex}

\bibliographystyle{alphanum}
\def\cprime{$'$}

\end{document}

%% file: Intro.tex
The weighted Radon transform is defined by
\begin{equation}
R_m[f](\xi,\eta)=\int_\R f(x,\xi x+\eta)m(x,\xi,\eta)\d x, \label{eq:WRT}
\end{equation}
for suitable functions $f=f(x,y)$ and $m=m(x,\xi,\eta)\geq0$. The case $m\equiv1$ corresponds to the ordinary Radon transform,
\begin{equation}
R[f](\xi,\eta)=\int_\R f(x,\xi x+\eta)\d x, \label{eq:RT}
\end{equation}
in which case its value at $(\xi,\eta)$ is given by the integral of $f$ over the line $\{(x,y);y=\xi x+\eta\}$. I.e.\ we integrate $f$ over the line with slope given by the $\xi$-parameter and intersecting the $y$-axis at $y=\eta$.
 
The question of invertibility of $R_m$ has been considered for a long time. Novikov \cite{Nov1} solved the problem when $m$ is an \emph{attenuation}, i.e.
\begin{equation}
m(x,\xi,\eta)=\exp\left(-\int_x^\infty \mu(t,\xi t+\eta)\d t\right),\label{eq:att}
\end{equation}
for compactly supported $\mu\in C^{0,\alpha}(\R^2)$. Independently, work by Arbuzov, Bukhgeim and Kazantsev could also be seen to imply an inversion formula \cite{ArBuKa}. After this was presented many related results and improvements followed, 
\cite{Natt1,BoStr,Bal1,BalTa,Finch1,KazBuk,Fokas,Gindikin,StrA}.

Strichartz, in \cite{Strichartz}, showed that the Radon transform is \emph{locally injective}, in the sense that if $\supp f\subset\{(x,y);y\geq x^2\}$ and the line integrals as defined in \eqref{eq:RT} are zero in a $(\xi,\eta)$-neighborhood of the origin, then $f=0$ in some neighborhood of the origin in the $(x,y)$-plane.

The corresponding statement unfortunately does not hold in general for the weighted Radon transform. In \cite{Bo3}, an example of a smooth $m\geq0$ is given for which $R_m$ is not locally injective. However, in the case when $m$ is real-analytic it is still true that the weighted Radon transform is locally injective, \cite{BoQu}. In \cite{Bo1} the class of weights for which the same conclusion holds was extended to smooth weights that satisfy an additional condition, first introduced by Gindikin, \cite{Gindikin}. See also \cite{Bo2,Bo4,MarkoeQuinto,Hertle} for more discussions on the local injectivity problem. It is however still not known how large the subspace of smooth weights is for which local injectivity holds.

In practice it is hard to verify local injectivity. Furthermore, no reasonable stability estimate (see further our example on Hölder versus logarithmic continuity at the end of section \ref{sec:summary}) will hold without an a priori bound on the data, even for constant weights. This is also the case for analytic continuation, \cite{FJohn}. Neither will even weaker Sobolev estimates be valid in the local problem, such as
\[
\|f\|_{H^s(\R^2)}\lesssim\|R[f]\|_{H^{0,t}(\T\times\R)},\quad t>s.
\]
The same thing is true for the so-called exterior problem. 

Recently, Caro, Ferreira and Ruiz in \cite{Ruiz}, Theorem 2.5(c), proved an estimate of very similar type as we are about to achieve, but only for the ordinary Radon transform. Bukhgeim has also made a contribution in this direction when $m$ is analytic, \cite{BuSuppl}. Finally we mention \cite{RQCRT} where Rullgård and Quinto presented quantitative Sobolev-type estimates with a remainder term.
\subsection{Main results}
\label{sec:summary}
We will consider the local stability problem for the weighted Radon transform as defined in \eqref{eq:WRT}. Assume that $f=f(x,y)$ satisfies the a priori bound $\|f\|_{C^{0,\alpha}(\R^2)}\leq C_0$ with $\supp f\subset\{(x,y);y\geq x^2\}$ (both of these conditions can be relaxed as we shall see, but doing so would introduce some rather unnecessary technicalities at this stage). The weight functions $m=m(x,\xi,\eta)$ will be assumed to be of very certain types. We always assume that for some functions $a=a(\xi,\eta)$ and $b=b(\xi,\eta)$, $m$ solves the partial differential equation
\begin{equation}
\del_\xi m(x,\xi,\eta)-x\del_\eta m(x,\xi,\eta)=(xa(\xi,\eta)+b(\xi,\eta))m(x,\xi,\eta).\label{eq:boman_cond}
\end{equation}
Condition \eqref{eq:boman_cond} also appears in \cite{Gindikin,Bo1,Bo4}. Weights that satisfy this condition can also be interpreted as attenuations, but in dual coordinates. We will consider two cases for the functions $a,b$. First that they are real analytic and then that they belong to a Gevrey space, $G_0^\sigma(\R^2)$. These spaces are defined in detail in the appendix, section \ref{sec:Gevrey}.

We will start by deriving estimates for certain means, $\mathcal{M}_{\eps,\gamma}[f]$ or $\mathcal{M}_{\eps,\gamma}[fm_\gamma]$, that we define in detail in section \ref{s:means}. From these estimates we are then able to deduce various estimates for $f$ or $fm_\gamma$.

Basically, for a fixed $x$, $\mathcal{M}_{\eps,\gamma}[fm_\gamma](x)$ is a mean of $fm_\gamma$ over the vertical interval $\{y;|y-\gamma|\leq\eps|x|\}$ defined by a convolution with a test function. The subscript $\gamma$ on the weight $m$ is a small technicality that indicates a certain correction that must be made in our arguments.

Since the intervals over which the means are computed will be small one can expect $\mathcal{M}_{\eps,\gamma}[fm_\gamma](x)$ to be close to $f(x,\gamma)m(x,0,\gamma)$. So from a practical viewpoint, also the estimates for these means are of some interest.
 
Our first result holds when $m$ is constant or $a,b$ are real analytic. It says that
\begin{equation}
\|\mathcal{M}_{\eps,\gamma}[fm_\gamma]\|_2\leq C\left(\frac{1}{\log(\|R_m[f]\|_{\eps,\gamma}^{-1})}\right)^{\alpha},\label{mainRT}
\end{equation}
where $\alpha>0$ is a constant depending on the regularity of $\mathcal{M}_{\eps,\gamma}[fm_\gamma]$, $\|\cdot\|_2$ denotes the usual $L^2$-norm and $\|\cdot\|_{\eps,\gamma}$ is a certain norm of the data. The constant $C>0$ will depend on the a priori bound of $\|f\|_{C^{0,\alpha}(\R^2)}$. 

We can conclude from \eqref{mainRT} that as $R_m[f]\to0$, $\mathcal{M}_{\eps,\gamma}[fm_\gamma]\to0$. From \eqref{mainRT} it is not very hard to derive the more interesting estimate
\begin{equation}
\|fm_\gamma\|_2\leq C\left(\frac{\log\log(\|R_m[f]\|_{\eps,\gamma}^{-1})}{\log(\|R_m[f]\|_{\eps,\gamma}^{-1})}\right)^{\alpha},\label{mainfRT}
\end{equation}
but sadly we get an additional $\log\log$-factor in this transition. 

For $\alpha>1/2$ we can also get a supremum estimate of the type in \eqref{mainRT}. However, the exponent $\alpha$ in the right hand side must be replaced by $\rho=\alpha-1/2$ and the norm on the left hand side must be replaced with a certain supremum norm of $f(\cdot,\gamma)m(\cdot,0,\gamma)$.

In the more general setting where $m$ is of Gevrey type, but still satisfies \eqref{eq:boman_cond}, we get the estimate
\begin{equation}
\|\mathcal{M}_{\eps,\gamma}[fm_\gamma]\|_2\leq C\left(\frac{\log\log(\|R_m[f]\|_{\eps,\gamma}^{-1})}{\log(\|R_m[f]\|_{\eps,\gamma}^{-1})}\right)^{\alpha},\label{mainWRT}
\end{equation}
where all constants fulfill similar conditions as in estimate \eqref{mainRT}. In an analogous way we can then get a supremum (for $\alpha>1/2$) or $L^2$-estimate of $fm_\gamma$. The estimate corresponding to \eqref{mainfRT} will be
\begin{equation}
\|fm_\gamma\|_2\leq C\left(\frac{\log^2\log(\|R_m[f]\|_{\eps,\gamma}^{-1})}{\log(\|R_m[f]\|_{\eps,\gamma}^{-1})}\right)^{\alpha}.\label{mainfWRT}
\end{equation}
Both \eqref{mainfRT}, \eqref{mainfWRT} can give us relevant information regarding the local injectivity question though. For example if $R_{m}[f]=0$ in some neighborhood of the origin and $m>0$ we would be able to conclude that $f=0$ in some neighborhood of the origin.

The outline of this paper will be that we will first define the means $\mathcal{M}_{\eps,\gamma}$ and derive some basic properties for them. Then we start by considering the case of the standard Radon transform \eqref{eq:RT}, where $m\equiv1$, and derive a stability estimate for it. This will illustrate the fundamental ideas involved in the proofs of \eqref{mainWRT} and \eqref{mainRT}.

The first key ingredient will be moment estimates of the means $\mathcal{M}_{\eps,\gamma}[fm_\gamma]$, summarized in Lemma \ref{l:mom0}. These will allow us to get estimates for coefficients in the expansions of the means in Fourier-Legendre series. The estimates will be of type
\begin{equation}
|a_n|\leq C^n\|R[f]\|,\label{constbds1}
\end{equation}
with some apropriate norm on the data $R[f]$. The a priori bound $\|f\|_{C^{0,\alpha}(\R^2)}\leq C_0$ will furthermore imply that $|a_n|\leq M/n^\alpha$, where $M$ depends on $C_0$. Using that the coefficients tend to zero in this way, together with \eqref{constbds1}, we will then be able to derive the desired estimates.
 
We would also like to mention that we have been, to some extent, inspired by arguments found in John's paper \cite{FJohn} on continuous dependence on data for solutions of partial differential equations with a prescibed bound. As a related remark, in the case of the Radon transform we do not expect any better than logarithmic continuity to describe the dependence between $f$ and the data $R[f]$. Without sketching all the details, consider for some arbitrary $\lambda>0$ and $q\in C_0^\infty(\R^2)$ the function
\[
f_\lambda(x,y)=q(x,y)\frac{\cos(\lambda x)}{\lambda}.
\]
Then one can verify that $\|f_\lambda\|_{C^{0,1}}\leq M=\sup|q|$, i.e. $f_\lambda$ satisfies a Lipschitz condition and $\|f_\lambda\|_2\approx\frac1\lambda$. The Radon transform of $f_\lambda$ is
\[
R[f_\lambda](\xi,\eta)=\frac1\lambda\int_\R q(x,\xi x+\eta)\cos(\lambda x)\d x.
\]
Doing $p-1$ integrations by parts (compare with \cite{Bo3,Bo2}) one can verify that
\[
|R[f_\lambda]|\leq C_p\lambda^{-p},
\]
holds for arbitrary integers $p$ and thus $\|R[f_\lambda]\|_2\approx C_p\lambda^{-p}$ (with a possibly new choice of constant $C_p$). So by choosing $p$ large enough we see that no inequality of the form
\[
\|f_\lambda\|_2\leq C\|R[f_\lambda]\|_2^\alpha,
\]
can hold for any $\alpha>0$.

%% file: Means.tex
\label{s:means}
Recall the definition of a function $f$ being Hölder continuous on an open subset $\Omega\in\R^n$ if there exists a $C_0>0$ and $0<\alpha\leq1$ such that for all $\bx,\by\in\Omega$,
\[
|f(\bx)-f(\by)|\leq C_0\|\bx-\by\|^\alpha.
\]
In the case $\alpha=1$, $f$ is called Lipschitz continuous. One defines the Hölder space $C^{0,\alpha}(\Omega)$, consisting of the complex-valued, Hölder continuous functions on $\Omega$. There is also an associated semi-norm
\[
\|f\|_{C^{0,\alpha}(\Omega)}:=\sup_{\substack{\bx,\by\in\Omega \\ \bx\neq\by}}\frac{|f(\bx)-f(\by)|}{\|\bx-\by\|^\alpha}.
\]

From functions $f$, or $fm$ for weights $m$, we can construct means $\mathcal{M}_{\eps,\gamma}[f]$ by convolving with a test function. Later we will make further restrictions on these test functions.

\begin{definition}\label{d:hor}
Suppose that $0\leq\varphi\in C_0^\infty[-1,1]$ is even and $\int\varphi=1$. Denote by $\varphi_{\eps|x|}(\xi)=\frac1{\eps|x|}\varphi\left(\frac\xi{\eps|x|}\right)$ where $\eps>0,x\neq0$. Then for continuous functions $f=f(x,y)$, we define for $x\neq0$,
\begin{equation}\label{mean1}
\mathcal{M}_{\eps,\gamma}[f](x)=\mathcal{M}_{\varphi,\eps,\gamma}[f](x)=f\circledast_y\varphi_{\eps|x|}(x,\gamma)=\int_{|y-\gamma|\leq\eps|x|} f(x,y)\varphi_{\eps|x|}(\gamma-y)\d y.
\end{equation}
For $x=0$, we choose to define $\mathcal{M}_{\eps,\gamma}[f](0)=f(0,\gamma)$.
\end{definition}
By $\circledast_y$ we mean convolution in the $y$-variable. Observe that, as a function of $(x,y)$, the support of $\varphi_{\eps|x|}(\gamma-y)$ is contained in the conic set $C_{\eps,\gamma}=\{(x,y);|y-\gamma|\leq\eps|x|\}$. The situation is illustrated in Figure \ref{fig:supp3}.
\begin{figure}[h!]
\centering
\includegraphics[scale=0.8]{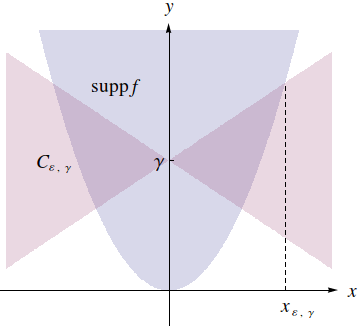}
\caption{The support of $f$ (blue parabola) and the set $C_{\eps,\gamma}$ (pink wedge) containing the supports of the cut-off functions $\varphi_{\eps|x|}(\gamma-y)$. The positive solution $x_{\eps,\gamma}$ of the equation $x^2=\eps x+\gamma$ implies thus $\supp\mathcal{M}_{\eps,\gamma}[f]\subset[-x_{\eps,\gamma},x_{\eps,\gamma}]$.}
\label{fig:supp3}
\end{figure}

When we also have a weight function $m=m(x,\xi,\eta)$ involved, the mean \eqref{mean1} need to be slightly modified to
\begin{equation}\label{mean2}
\mathcal{M}_{\eps,\gamma}[fm_\gamma](x)=fm_\gamma\circledast_y\varphi_{\eps|x|}(x,\gamma)=\int_{|y-\gamma|\leq\eps|x|} f(x,y)m_\gamma(x,y)\varphi_{\eps|x|}(\gamma-y)\d y,
\end{equation}
and for $x=0$, we set $\mathcal{M}_{\eps,\gamma}[fm_\gamma](0)=f(0,\gamma)m(0,0,\gamma)$. Here 
\begin{equation}
m_\gamma(x,y)=m\left(x,\frac{y-\gamma}{x},\gamma\right),\quad x\neq0.\label{m_gamma}
\end{equation}
That we have to modify the weight function $m$ in this way will become clear later when we see how these means appear when we consider the weighted Radon transform.

\subsection{Convergence and regularity}
In this part we will prove some convergence and regularity results for $\mathcal{M}_{\eps,\gamma}[f]$, but the statements also hold, with only minor changes, for $\mathcal{M}_{\eps,\gamma}[fm_\gamma]$.

First observe that the support of $\mathcal{M}_{\eps,\gamma}[f]$ is for fixed $\gamma\in\R$ contained within the interval $\{x;|x|\leq x_{\eps,\gamma}\}$, see Figure \ref{fig:supp3}. We will assume that $\eps>0$ and $\gamma>0$ are so small that $x_{\eps,\gamma}\leq1$ and consider $\mathcal{M}_{\eps,\gamma}[f]$ as a function over $[-1,1]$. $\mathcal{M}_{\eps,\gamma}[f](x)\to f(x,\gamma)$ uniformly on compact subsets as $\eps\to0$ since $f$ is continuous. Proposition \ref{prop:pconv} gives a result on the rate of convergence when $f$ is Hölder continuous.

\begin{proposition}\label{prop:pconv}
Suppose $\|f\|_{C^{0,\alpha}(\R^2)}\leq C_0$, then for every $\gamma\in\R, x\in\R, \eps>0$,
\[
|\mathcal{M}_{\eps,\gamma}[f](x)-f(x,\gamma)|\leq C_0(\eps|x|)^\alpha\leq C_0\eps^\alpha.
\]
For smooth weights $m$, the same conclusion holds for  $fm_\gamma$ in place of $f$.
\end{proposition}

\begin{proof}
Since $\varphi_{\eps|x|}(\gamma-y)$ is an approximation to the identity
\begin{multline*}
|\mathcal{M}_{\eps,\gamma}[f](x)-f(x,\gamma)|=\left|\int_{|y-\gamma|\leq\eps|x|}[f(x,y)-f(x,\gamma)]\varphi_{\eps|x|}(\gamma-y)\d y\right|\\
\leq\int|f(x,y)-f(x,\gamma)|\varphi_{\eps|x|}(\gamma-y)\d y\leq C_0\int|y-\gamma|^\alpha\varphi_{\eps|x|}(\gamma-y)\d y\\
\leq C_0(\eps|x|)^\alpha\int \varphi_{\eps|x|}(\gamma-y)\d y=C_0(\eps|x|)^\alpha\leq C_0\eps^\alpha,
\end{multline*}
since
\[
\supp f(x,y)\varphi_{\eps|x|}(\gamma-y)\subset\{(x,y);|y-\gamma|\leq\eps|x|,|x|\leq x_{\eps,\gamma}\leq1\}.
\]
\end{proof}
One also has that if $f\in C^{0,\alpha}(\R^2)$, then $\mathcal{M}_{\eps,\gamma}[f]\in C^{0,\alpha}(\R)$ by the following theorem:

\begin{theorem}\label{t:preg}
If $\|f\|_{C^{0,\alpha}(\R^2)}\leq C_0$, then for every $\gamma\in\R, 0<\eps<1$ the mean values $\mathcal{M}_{\eps,\gamma}[f]$ belong to $C^{0,\alpha}(\R)$ and
\[
|\mathcal{M}_{\eps,\gamma}[f](x)-\mathcal{M}_{\eps,\gamma}[f](x')|\leq \sqrt{2}C_0(|x-x'|)^\alpha.\label{ineqpreg}
\]
For smooth weights $m$, the same conclusion holds for $fm_\gamma$ in place of $f$.
\end{theorem}

\begin{proof}
The case $x=0$ or $x'=0$ follows from Proposition \ref{prop:pconv}, so consider the case $x\neq0\neq x'$. By changing variables
\begin{multline*}
|\mathcal{M}_{\eps,\gamma}[f](x)-\mathcal{M}_{\eps,\gamma}[f](x')|=\\
=\left|\int_{\R}f(x,y)\frac1{\eps|x|}\varphi\left(\frac{y-\gamma}{\eps|x|}\right)-f(x',y)\frac1{\eps|x'|}\varphi\left(\frac{y-\gamma}{\eps|x'|}\right)\d y\right|\\
=\left|\int_{-1}^1f(x,\eps|x|t+\gamma)\varphi(t)-f(x',\eps|x'|t+\gamma)\varphi(t)\d t\right|\\
\leq\int_{-1}^1\big|f(x,\eps|x|t+\gamma)-f(x',\eps|x'|t+\gamma)\big|\varphi(t)\d t\\
\leq C_0\int_{-1}^1\|(x-x',\eps t(x-x'))\|^\alpha\varphi(t)\d t\leq C_0(\sqrt{1+\eps^2}|x-x'|)^\alpha\leq\sqrt{2}C_0(|x-x'|)^\alpha.
\end{multline*}
\end{proof}
\begin{remark}
In the proof of Proposition \ref{prop:pconv}, we really only used Hölder continuity in the $y$-variable, and in the proof of Theorem \ref{t:preg} we could manage with a uniform Hölder condition along lines with slope smaller than $\eps>0$. So these are obvious relaxations that can be made in the statements. 
\end{remark}

%% file: RT.tex
\label{sec:1}
In this section we will prove the stability estimates \eqref{mainRT} and \eqref{mainfRT} for constant $m$. The standard Radon (or X-ray) transform in the plane is defined by
\begin{equation}
R[f](\xi,\eta)=\int_\R f(x,\xi x+\eta)\d x,\label{def:RT}
\end{equation}
for suitable functions $f=f(x,y)$. We will assume that $\supp f\subset\{(x,y);y\geq x^2\}$ and that an a priori bound on $f$ of type $\|f\|_{C^{0,\alpha}(\R^2)}\leq C_0$ holds (or a similar Hölder condition for all lines with small slope). The former assumption can be seen to imply that $\supp R[f]\subset\{(\xi,\eta);\eta\geq-\frac{\xi^2}4\}$, c.f.\ Figure \ref{fig:supp1}. 

The dual Radon transform is defined by 
\begin{equation}
R^*[\varphi](x,y)=\int_\R\varphi(\xi,y-\xi x)\d\xi\label{def:DRT}
\end{equation}
for suitable functions $\varphi=\varphi(\xi,\eta)$. If we for the moment only assume that $\varphi$ is such that $\supp\varphi\, \cap\supp R[f]$ and $\supp R^*[\varphi]\, \cap\supp f$ is compact one can easily show that $R^*$ is the proper adjoint of $R$, that is
\[
\langle R[f],\varphi\rangle=\langle f,R^*[\varphi]\rangle
\]
whenever $R[f]\varphi\in L_{\text{loc}}^1(\R^2)$ and $fR^*[\varphi]\in L_{\text{loc}}^1(\R^2)$.

Observe now that if $f$ would be sufficiently regular
\[
\del_\eta R[xf](\xi,\eta)=\int_\R x\,\del_y f(x,\xi x+\eta)\d x=\int_\R\del_\xi f(x,\xi x+\eta)\d x=\del_\xi R[f](\xi,\eta),
\]
where by $\del_\eta$ we mean partial derivatives $\frac{\del}{\del\eta}$ etc. Iterating this gives the important identity
\begin{equation}
\del_\eta^kR[x^kf]=\del_\xi^kR[f],\label{id:1}
\end{equation}
which also holds in the sense of distributions. Hence it will be applicable also in our case.

Before moving on we would like to emphasize that the basic ideas will be the same also when we introduce a weight. Achieving corresponding moment estimates will however be more tricky due to that an identity corresponding to  \eqref{id:1} will not hold. 

\subsection{Moment estimates I}\label{sec:1.1}
A key ingredient in our proofs will be certain moment estimates. Consider first, 
\[
R[x^kf](\xi,\eta)=\int_\R x^kf(x,\xi x+\eta)\d x.
\]
Observe that for $k=1$, using \eqref{id:1}, we have
\begin{multline*}
R[xf](\xi,\eta)=\int_{-\infty}^\eta \del_{\eta'}R[xf](\xi,\eta')\d\eta'=\int_\R H(\eta-\eta')\del_\xi R[f](\xi,\eta')\d\eta'\\
=(H\circledast_\eta\del_\xi R[f])(\xi,\eta),
\end{multline*}
where $\circledast_\eta$ denotes convolution in the $\eta$-variable, $H$ denotes the Heaviside function. We will generalize the above identity using the fact that convolution of the Heaviside function with itself $k$ times results in the function $H_{k+1}(\eta)=\eta_+^k/k!$ (i.e. $H_1=H, H_2=H*H$, etc.), so
\begin{multline*}
(H_k\circledast_\eta\del_\xi^kR[f])(\xi,\eta)=\int_{-\infty}^\eta\frac{(\eta-\eta')^{k-1}}{(k-1)!}\del_\xi^{k}R[f](\xi,\eta')\d\eta'\\
=\int_{-\infty}^\eta\frac{(\eta-\eta')^{k-1}}{(k-1)!}\del_{\eta'}^{k}R[x^{k}f](\xi,\eta')\d\eta'\\
=\int_{-\infty}^\eta\del_{\eta'}R[x^{k}f](\xi,\eta')\d\eta'=R[x^{k}f](\xi,\eta).
\end{multline*}
To summarize, the following identity holds:
\begin{equation}
H_k\circledast_\eta\del_\xi^{k}R[f]=R[x^{k}f], \label{id:2}
\end{equation}
for $k=1,2,\dots$ (if we set $H_0(\eta)=\delta(\eta)$, Dirac's delta distribution, \eqref{id:2} also makes sense for $k=0$).

Now suppose $\varphi$ is a test function of the type that is mentioned in Definition \ref{d:hor}, and that we fix an $\eta=\gamma$. We can then derive the following identity, by just doing a change of variables $\xi$ to $y=\xi x+\gamma$,
\[
\iint_{\R^2}x^kf(x,\xi x+\gamma)\d x\,\varphi_{\eps}(\xi)\d\xi=\iint_{\R^2}x^kf(x,y)\varphi_{\eps|x|}(\gamma-y)\d x\d y.
\]
By $\varphi_{\eps}(\xi)=\frac1\eps\varphi\left(\frac{\xi}{\eps}\right)$ etc.\ we denote the usual dilations. Observe now that the $y$-integral is what defines the mean $\mathcal{M}_{\eps,\gamma}[f](x)$, so we have actually shown:
\begin{equation}\label{eq:p_eg}
\int_{|\xi|\leq\eps}R[x^kf](\xi,\gamma)\varphi_{\eps}(\xi)\d\xi=\int_{|x|\leq1}x^k\mathcal{M}_{\eps,\gamma}[f](x)\d x.
\end{equation}
This illustrates how the means $\mathcal{M}_{\eps,\gamma}[f]$ fit into our framework. If we introduce the moment functionals
\[
m_k[f]=\int x^kf(x)\d x
\]
we are ready to prove the following lemma:

\begin{lemma}\label{l:mom0}
Suppose $\mathcal{M}_{\eps,\gamma}[f]$ is as in Definition \ref{d:hor}, $\supp f\subset\{(x,y);y\geq x^2\}$ and that $\gamma\geq\eps^2/4$. Then for $k=1,2,\dots$,
\[ 
|m_k(\mathcal{M}_{\eps,\gamma}[f])|\leq \frac{(2\gamma)^{k-1}}{(k-1)!}\|\del_\xi^k\varphi_{\eps}\|_\infty\|R[f]\|_{L^1(R_{\eps,\gamma})}
\]
where $R[f](\xi,\eta)=\int f(x,\xi x+\eta)\d x, \|R[f]\|_{L^1(R_{\eps,\gamma})}=\iint_{R_{\eps,\gamma}}|R[f](\xi,\eta)|\d\xi\d\eta$ and $R_{\eps,\gamma}=\{(\xi,\eta);|\xi|\leq\eps,|\eta|<\gamma\}$.
\end{lemma}
\begin{proof}
\[
m_k(\mathcal{M}_{\eps,\gamma}[f])=\int_{-\eps}^\eps R[x^kf](\xi,\gamma)\varphi_{\eps}(\xi)\d\xi.
\]
Using \eqref{id:2}, doing $k$ integrations by parts and assuming that $\gamma\geq\eps^2/4$, we get,
\begin{multline*}
|m_k(\mathcal{M}_{\eps,\gamma}[f])|=\left|\int_{-\eps}^\eps\int_{-\gamma}^{\gamma}\frac{(\gamma-\eta)^{k-1}}{(k-1)!}R[f](\xi,\eta)\varphi_{\eps}^{(k)}(\xi)\d\eta\d\xi\right|\\
\leq\frac{(2\gamma)^{k-1}}{(k-1)!}\|\del_\xi^k\varphi_{\eps}\|_\infty\iint_{R_{\eps,\gamma}}|R[f](\xi,\eta)|\d\xi\d\eta.
\end{multline*}
\end{proof}
\begin{remark}
The condition $\gamma\geq\eps^2/4$ could be removed by just replacing $\gamma$ everywhere by $\max\{\gamma,\eps^2/4\}$ (c.f.\ Figure \ref{fig:supp1}). For brevity we choose to keep this assumption in what follows.
\end{remark}
Applying Lemma \ref{l:mom0}, with a test function $\varphi=\phi\in G_0^\sigma([-1,1])$ (so that $\|\del_\xi^k\phi\|_\infty\leq Ck!^\sigma, k=1,2,\dots$, c.f.\ Appendix, Section \ref{sec:Gevrey}) in the definition of $\mathcal{M}_{\eps,\gamma}=\mathcal{M}_{\varphi,\eps,\gamma}$ we get:
\[
|m_k(\mathcal{M}_{\eps,\gamma}[f])|\leq\frac{(2\gamma)^{k-1}C^{k+1}k!^s}{\eps^{k+1}}\|R[f]\|_{L^1(R_{\eps,\gamma})},\quad k=1,2,\dots.
\]
For the most general weight functions that we will consider in section \ref{s:gevwts} we are going to have to do a similar argument as the one above. 

When working with constant (or real analytic) $m$ we can however do better. We can choose a test function $\varphi$ from a special sequence detailed in the Appendix, Section \ref{s:horseq}. Denote such a test function by $\varphi_N$, where $N$ is an integer to be determined later, and its derivatives can be estimated by
\[
\|\del_\xi^k\varphi_{N,\eps}\|_\infty=\frac{1}{\eps^{k+1}}\|\varphi_N^{(k)}\|_\infty\leq\frac{C^{k+1}N^k}{\eps^{k+1}},\quad k\leq N.
\]
Using this we get the following special case of Lemma \ref{l:mom0}:
\begin{lemma}\label{p:mom1}
Suppose $\mathcal{M}_{\eps,\gamma}[f]$ is as in Definition \ref{d:hor}, then
\[
|m_k(\mathcal{M}_{\eps,\gamma}[f])|\leq\frac{(2\gamma)^{k-1}C^{k+1}N^k}{(k-1)!\eps^{k+1}}\|R[f]\|_{L^1(R_{\eps,\gamma})},\quad k=1,2,\dots,N,
\]
where $R[f](\xi,\eta)=\int f(x,\xi x+\eta)\d x$ and $R_{\eps,\gamma}=\{(\xi,\eta);|\xi|\leq\eps,|\eta|<\gamma\}$.
\end{lemma}
\begin{figure}[h!]
\centering
\includegraphics[scale=0.8]{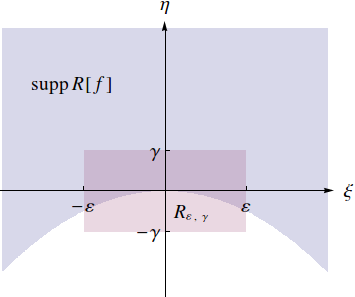}
\caption{The support of $R[f]$ (blue set, complementary to parabola) and the set $R_{\eps,\gamma}$ (pink rectangle) containing the support of $R[f](\xi,\eta)(\gamma-\eta)_+^k\varphi_{\eps}(\xi)$ under the assumption $\gamma\geq\eps^2/4$. If the last condition is not fulfilled $R_{\eps,\gamma}=\{(\xi,\eta);|\xi|\leq\eps,-\max(\gamma,\eps^2/4)\leq\eta\leq\gamma\}$.}
\label{fig:supp1}
\end{figure}
\begin{remark}\label{p:mom0}
Observe that for $k=0$ one has by similar arguments
\[
|m_0(\mathcal{M}_{\eps,\gamma}[f])|\leq\|\varphi_{N,\eps}\|_\infty\int_{-\eps}^\eps|R[f](\xi,\gamma)|\d\xi\leq\frac{C}{\eps}\|R[f](\cdot,\gamma)\|_{L^1[-\eps,\eps]}.
\]
This estimate is not of the same type as in Lemma \ref{p:mom1} due to the different norm on the data in the right hand side. A simple workaround is to define another norm $\|\cdot\|_{\eps,\gamma}$, increasing in $\eps$ and $\gamma$, such that
\begin{equation}\label{newnorm}
\|R[f]\|_{\eps,\gamma}\gtrsim\max\{\|R[f]\|_{L^1(R_{\eps,\gamma})},\|R[f](\cdot,\gamma)\|_{L^1[-\eps,\eps]}\}
\end{equation}
(recall that $f\lesssim g$ means that there is an absolute constant $c>0$ such that $f\leq cg$).
An example would be
\[
\|R[f]\|_{\eps,\gamma}=\sup_{|\eta|\leq\gamma}\|R[f](\cdot,\eta)\|_{L^1[-\eps,\eps]}.
\]
\end{remark}

Now we can easily combine Lemma \ref{p:mom1} and Remark \ref{p:mom0} into a proposition that bounds all $k$-moments for $k=0,1,\dots,N$.
\begin{proposition}\label{p:mom}
Suppose $\mathcal{M}_{\eps,\gamma}[f]$ is as in Definition \ref{d:hor}, $\supp f\subset\{(x,y);y\geq x^2\}$ and that $\eps^2/4\leq\gamma<1$, then for $k=0,1,2,\dots,N$
\[
|m_k(\mathcal{M}_{\eps,\gamma}[f])|\leq\left(\frac{C}{\eps}\right)^{k+1}\e^N\|R[f]\|_{\eps,\gamma}.
\]
\end{proposition}

\begin{proof}
Since $k\leq2^{(k+1)/2}$ for $k\geq0$ and by Lemma \ref{p:mom1},
\[
\frac{(2\gamma)^{k}C^{k+1}N^k}{(k-1)!\eps^{k+1}}\leq\left(\frac{2\sqrt2 C}{\eps}\right)^{k+1}\frac{N^k}{k!}.
\]
Using that $N^k/k!$ is simply a term in the series defining $\e^N$ and replacing $2\sqrt2 C$ with $C$ we complete the proof.
\end{proof}

\subsection{First stability estimate}
\label{s:rt_est}
Now we are ready to estimate the $L^2$-norm of $\mathcal{M}_{\eps,\gamma}[f]$.

\begin{theorem}\label{main2}
If $0<\eps<1, \eps^2/4\leq\gamma<1$, $\|f\|_{C^{0,\alpha}(\R^2)}\leq C_0$ and $\supp f\subset\{(x,y);y\geq x^2\}$, then
\begin{equation}
\|\mathcal{M}_{\eps,\gamma}[f]\|_2\leq4M\left(\frac{\log(C/\eps)}{\log(M/\|R[f]\|_{\eps,\gamma})}\right)^\alpha
\end{equation}
for small $\|R[f]\|_{\eps,\gamma}$, where $M$ depends on $C_0$.
\end{theorem}

\begin{proof}
By approximating $\mathcal{M}_{\eps,\gamma}[f]$ with a Fourier-Legendre sum $S_N[\mathcal{M}_{\eps,\gamma}[f]]$ on the interval $[-1,1]$, and using the triangle inequality we first get
\begin{equation}
\|\mathcal{M}_{\eps,\gamma}[f]\|_2\leq\left(\sum_{n=0}^N|a_n|^2\right)^{1/2}+\|\mathcal{M}_{\eps,\gamma}[f]-S_N[\mathcal{M}_{\eps,\gamma}[f]]\|_2.\label{eq_l2}
\end{equation}
By Lemma \ref{l:tailest} and Theorem \ref{t:preg}, we have that
\[
\|\mathcal{M}_{\eps,\gamma}[f]-S_N[\mathcal{M}_{\eps,\gamma}[f]]\|_2\leq\frac{4A_0C_0}{N^\alpha}=:\frac{M}{N^\alpha}.
\]
Next, by Lemma \ref{l:legmom} and Proposition \ref{p:mom}, we have for $n\leq N$
\begin{multline*}
|a_n|\leq(4\sqrt2)^n\max_{0\leq k\leq n}\left(\frac{C}{\eps}\right)^{k+1}\e^N\|R[f]\|_{\eps,\gamma}
=(4\sqrt2)^n\e^N\|R[f]\|_{\eps,\gamma}\left(\frac{C}{\eps}\right)^{n+1}.
\end{multline*}
It follows, by absorbing lower order factors into exponentials and choosing $C$ larger, that
\[
|a_n|\leq\left(\frac{C}{\eps}\right)^{n+1}\e^N\|R[f]\|_{\eps,\gamma},\quad n\leq N.
\]
Substituting this into the finite sum in \eqref{eq_l2} we get (again with a larger constant $C$)
\[
\left(\sum_{n=0}^N|a_n|^2\right)^{1/2}\leq\left(\frac{C}{\eps}\right)^{N+1}\|R[f]\|_{\eps,\gamma}.
\]
Since we assume $H=\|R[f]\|_{\eps,\gamma}$ is small we find $N$ such that
\begin{equation}
\left(\frac{C}{\eps}\right)^{N+1}H\leq\frac{M}{N^\alpha}.\label{eq:HN}
\end{equation}
Taking logarithms, the above is equivalent with
\[
(N+1)\log(C/\eps)+\alpha\log(N)\leq\log(M/H).
\]
The above condition is implied by
\[
(N+1)\left(\log(C/\eps)+\alpha\frac{\log(N)}{N+1}\right)\leq(N+1)\log(C/\eps)\leq\log(M/H),
\]
where the first inequality follows from choosing a new and slightly larger constant $C$ in the right hand side. The last conditions can be implied by choosing $N$ as the largest possible integer that satisfies:
\[
N\leq\frac{\log(M/H)-\log(C/\eps)}{\log(C/\eps)}.
\]
By doing so we get from \eqref{eq:HN} and \eqref{eq_l2},
\[
\|\mathcal{M}_{\eps,\gamma}[f]\|_2\leq\frac{2M}{N^\alpha},
\]
and furthermore
\[
N\geq\frac{\log(M/H)-2\log(C/\eps)}{\log(C/\eps)}
\]
can be chosen, so 
\[
\|\mathcal{M}_{\eps,\gamma}[f]\|_2\leq2M\left(\frac{\log(C/\eps)}{\log(M/H)-2\log(C/\eps)}\right)^\alpha.
\]
Assuming that $H$ is so small that
\[
\log(C/\eps)\leq\frac{1-\beta}{2}\log(M/H)
\]
for some $1/2\leq\beta<1$ we finally get
\[
\|\mathcal{M}_{\eps,\gamma}[f]\|_2\leq4M\left(\frac{\log(C/\eps)}{\log(M/H)}\right)^\alpha.
\]
If $H=\|R[f]\|_{\eps,\gamma}=0$ it is immediate from Proposition \ref{p:mom} that $\|\mathcal{M}_{\eps,\gamma}[f]\|_2=0$.
\end{proof}

While Theorem \ref{main2} is interesting in itself we can easily get a stability estimate involving $f$ by combining the above with Proposition \ref{prop:pconv}.
\begin{theorem}\label{f_stability}
Suppose that $\|f\|_{C^{0,\alpha}(\R^2)}\leq C_0, \supp f\subset\{(x,y);y\geq x^2\}$ and that $\gamma>0$ is small enough. Given any $\eps>0$ there exist $M>0$ depending on $C_0$ such that,
\[
\|f(\cdot,\gamma)\|_{2}\lesssim M\frac{\log^\alpha\log(\|R[f]\|_{\eps,\gamma}^{-1})}{\log^\alpha(\|R[f]\|_{\eps,\gamma}^{-1})},
\]
if $\|R[f]\|_{\eps,\gamma}$ is sufficiently small.
\end{theorem}
\begin{proof}
Using Proposition \ref{prop:pconv}, and Theorem \ref{main2} we may estimate the terms in the right hand side of the inequality
\begin{equation}
\|f(\cdot,\gamma)\|_2\leq\|f(\cdot,\gamma)-\mathcal{M}_{\eps,\gamma}[f]\|_2+\|\mathcal{M}_{\eps,\gamma}[f]\|_2\label{startineq}
\end{equation}
by 
\begin{equation}
\|\mathcal{M}_{\eps,\gamma}[f]\|_2\leq4M\left(\frac{\log(C/\eps)}{\log(M/\|R[f]\|_{\eps,\gamma})}\right)^\alpha\label{helpineq1}
\end{equation}
and
\begin{equation}
\|f(\cdot,\gamma)-\mathcal{M}_{\eps,\gamma}[f]\|_2\leq C_0\eps^\alpha.\label{helpineq2}
\end{equation}
Suppose now that we take $0<\eps<\eps_0<1$, $H_\eps=\|R[f]\|_{\eps,\gamma}\leq\|R[f]\|_{\eps_0,\gamma}=H$ and $M\geq1$, then it follows from \eqref{helpineq1}
\begin{equation}
\|\mathcal{M}_{\eps,\gamma}[f]\|_2\leq4M\left(\frac{\log(C/\eps)}{\log(M/H_\eps)}\right)^\alpha\leq4M\left(\frac{\log{C}+\log(1/\eps)}{\log(M/H)}\right)^\alpha.\label{helpineq3}
\end{equation}
Substituting \eqref{helpineq2} and \eqref{helpineq3} back into \ref{startineq}, we have
\begin{equation}
\|f(\cdot,\gamma)\|_2\leq4M\left(\frac{\log{C}+\log(1/\eps)}{\log(M/H)}\right)^\alpha+C_0\eps^\alpha,\quad \eps<\eps_0.
\end{equation}
Under the assumption that $H=\|R[f]\|_{\eps_0,\gamma}$ is sufficiently small we may choose $\eps$ such that
\[
\frac{1}{\log(M/H)}\leq\eps\leq\frac{2}{\log(M/H)}<\eps_0.
\]
Then
\[
\|f(\cdot,\gamma)\|_2\leq4M\left(\frac{\log{C}+\log\log(M/H)}{\log(M/H)}\right)^\alpha+C_0\left(\frac{2}{\log(M/H)}\right)^\alpha
\]
So clearly
\[
\|f(\cdot,\gamma)\|_2\lesssim M\left(\frac{\log\log(\|R[f]\|_{\eps_0,\gamma}^{-1})}{\log(\|R[f]\|_{\eps_0,\gamma}^{-1})}\right)^\alpha.
\]
In the case $\|R[f]\|_{\eps_0,\gamma}=0$ it immediately follows for all $\eps<\eps_0$ that $\|f(\cdot,\gamma)\|_2\leq C_0\eps^\alpha$. So we conclude in this case that $\|f(\cdot,\gamma)\|_2=0$.
\end{proof}

\subsection{Estimates involving $\sup|{\mathcal{M}_{\eps,\gamma}[f]}(x)|$ and $\sup|f(x,\gamma)|$}
In the step from Theorem \ref{main2} to Theorem \ref{f_stability}, we were able to get estimates of $f$ instead of just estimates of the means $\mathcal{M}_{\eps,\gamma}[f]$. But we apparently lose a little bit due to the added $\log\log$-factor in the nominator. However, in the case $\alpha>1/2$ we can even get supremum norm estimates without the added $\log\log$-factor, c.f.\ Corollary \ref{f_stability_var}. Simply observe that by Lemma \ref{l:szbd} we can find a constant $B$ such that the $L^2$-normalized Legendre polynomials over $[-1,1]$ satisfy
\begin{equation}
\sup_{|x|\leq1/2}|\tilde{P}_n(x)|\leq B.\label{LP_ubd}
\end{equation}
For example $B=2^{1/4}\sqrt{3/\pi}$ will do. It follows that
\[
\sup_{|x|\leq1/2}|\mathcal{M}_{\eps,\gamma}[f](x)|\leq\sum_{n=0}^\infty\sup_{|x|\leq1/2}|a_n\tilde{P}_n(x)|\leq B\sum_{n=0}^\infty(1+n)^{-\beta}|a_n|(1+n)^\beta.
\]
By Cauchy-Schwarz inequality, in the case $\beta>1/2$, we get
\begin{multline*}
\sup_{|x|\leq1/2}|\mathcal{M}_{\eps,\gamma}[f](x)|^2\leq B^2\sum_{n=0}^\infty(1+n)^{-2\beta}\sum_{n=0}^\infty(1+n)^{2\beta}|a_n|^2\\=B^2\zeta(2\beta)\sum_{n=0}^\infty(1+n)^{2\beta}|a_n|^2,
\end{multline*}
where $\zeta(s)$ is the Riemann zeta function. Now we split the last sum on the right hand side and observe that we may estimate
\[
\left(\sum_{n=0}^{N}(1+n)^{2\beta}|a_n|^2\right)^{1/2}\leq\left(\frac{C}{\eps}\right)^{N+1}\|R[f]\|_{\eps,\gamma}
\]
just as in the proof of Theorem \ref{main2} since we only have an extra lower order factor. For the tail we can do a summation by parts,
\[
\sum_{n=N+1}^{\infty}(1+n)^{2\beta}|a_n|^2=A_{N+1}(1+(N+1))^{2\beta}+\sum_{n=N+2}^{\infty}A_n((1+n)^{2\beta}-n^{2\beta}),
\]
where we know from Lemma \ref{l:tailest} that 
\[
A_{N+1}=\sum_{k=N+1}^\infty|a_k|^2\leq\frac{M^2}{N^{2\alpha}}.
\]
Hence the first term can be estimated by
\[
A_{N+1}(N+2)^{2\beta}\leq\frac{M^2}{N^{2(\alpha-\beta)}}\left(1+\frac{2}{N}\right)^{2\beta}\lesssim\frac{M^2}{N^{2(\alpha-\beta)}}.
\]
The second series can be estimated by
\begin{multline*}
\sum_{n=N+2}^{\infty}A_n((1+n)^{2\beta}-n^{2\beta})\leq\sum_{n=N+2}^\infty\frac{M^2}{(n-1)^{2\alpha}n^{-2\beta}}\left(\left(1+\frac{1}{n}\right)^{2\beta}-1\right)\\
\leq\sum_{n=N+2}^\infty\frac{3M^2}{(n-1)^{2\alpha}n^{-2\beta+1}}\lesssim\sum_{n=N+1}^\infty\frac{M^2}{n^{2(\alpha-\beta)+1}}\lesssim\frac{M^2}{N^{2(\alpha-\beta)}}.
\end{multline*}
So we get,
\[
\left(\sum_{n>N}(1+n)^{2\beta}|a_n|^2\right)^{1/2}\lesssim\frac{M}{N^{\alpha-\beta}},
\]
for arbitrary $\alpha>\beta$. Again, by the same arguments as in the proof of Theorem \ref{main2} one arrives at
\begin{theorem}\label{main2var}
Suppose that $0<\eps<1$, $\gamma>0$ is small enough, $\|f\|_{C^{0,\alpha}(\R^2)}\leq C_0$, where the exponent $\alpha>1/2$, $0<\rho<\alpha-1/2$, and $\supp f\subset\{(x,y);y\geq x^2\}$. Then
\begin{equation}
\sup_{|x|\leq1/2}|\mathcal{M}_{\eps,\gamma}[f](x)|\lesssim M\left(\frac{\log(C/\eps)}{\log(M/\|R[f]\|_{\eps,\gamma})}\right)^\rho,
\end{equation}
for small $\|R[f]\|_{\eps,\gamma}$ and $M$ depending on $C_0$.
\end{theorem}
In particular we get from Theorem \ref{main2var} the estimate
\begin{equation}
|f(0,\gamma)|\leq\frac{C}{\log^\rho(1/\|R[f]\|_{\eps,\gamma})}\label{eq:cor_var}
\end{equation}
for some $C>0$ (depending on $C_0$, for some fixed $\eps>0$). To get supremum norm estimates for other points on a line $y=\delta$, for some $0<\delta<\gamma$ we can apply Theorem \ref{main2var} on $f(x,y)=u(x+b,y-\delta)$ where $\supp u$ is contained in a more narrow parabola (e.g.\ $\supp u\subset\{(x,y);y\geq2x^2\}$). (Or we could of course have considered $\supp f$ to be in a wider parabola from the start and proved Theorem \ref{main2var} with obvious modifications.) Assume that also $u=u(x,y)$ satisfies the a priori estimate $\|u\|_{C^{0,\alpha}}\leq C_0$ for the same $\alpha>1/2$. Now, since for any real number $b$, $2(x+b)^2+\delta\geq x^2$ holds if $\delta\geq2b^2$, assuming that $|b|\leq\sqrt{\delta/2}$ it can be seen to follow from \eqref{eq:cor_var}, that
\begin{equation}
|u(b,\delta)|\leq\frac{C}{\log^\rho(1/\|R[u]\|_{\eps,\delta})},\quad|b|\leq\sqrt{\delta/2}.\label{eq:cor_var2}
\end{equation}
Hence we have shown
\begin{corollary}\label{f_stability_var}
If $u=u(x,y)$ is a function such that $\|u\|_{C^{0,\alpha}(\R^2)}\leq C_0$, where the exponent $\alpha>1/2$, $0<\rho<\alpha-1/2$ and $\supp u\subset\{(x,y);y\geq cx^2\}$ for some $c>1$, then
\[
\sup_{x}|u(x,\delta)|\leq\frac{C}{\log^\rho(1/\|R[u]\|_{\eps,\delta})}
\]
for some $\eps>0, \delta>0$, $\|R[u]\|_{\eps,\delta}$ small and $C$ depending on $C_0$.
\end{corollary}
\begin{remark}
Observe that in Corollary \ref{f_stability_var} we actually only require that a Hölder condition is fulfilled along all lines with slope smaller than $\eps>0$.
\end{remark}

%% file: WRT.tex
Now we move on towards a similar stability estimate for the weighted Radon transform. Assume that $m\in C^\infty(\R^3)$, and that $f\in C^{0,\alpha}(\R^2)$ with $\supp f\subset\{(x,y);y\geq x^2\}$. Then we define the weighted Radon transform (with weight $m$) of $f$ by
\begin{equation}
R_m[f](\xi,\eta)=\int_\R f(x,\xi x+\eta)m(x,\xi,\eta)\d x,\label{def:Rm}
\end{equation}
and its adjoint by
\begin{equation}
R_m^*[g](x,y)=\int_\R g(\xi,y-\xi x)m(x,\xi,y-\xi x)\d\xi.\label{def:Rm*}
\end{equation}
Observe that
\begin{align}
\del_\xi R_m[f](\xi,\eta)=&\int_\R x\,\del_yf(x,\xi x+\eta)m(x,\xi,\eta)+f(x,\xi x+\eta)\del_\xi m(x,\xi,\eta)\d x\label{eq:d1Rm}\\
\del_\eta R_m[xf](\xi,\eta)=&\int_\R x\,\del_yf(x,\xi x+\eta)m(x,\xi,\eta)+xf(x,\xi x+\eta)\del_\eta m(x,\xi,\eta)\d x\label{eq:d2Rm}.
\end{align}
So obviously we do not have an identity as \eqref{id:1}. However, subtracting \eqref{eq:d2Rm} from \eqref{eq:d1Rm} we get
\[
\del_\xi R_m[f](\xi,\eta)-\del_\eta R_m[xf](\xi,\eta)=\int_\R f(x,\xi x+\eta)[\del_\xi m(x,\xi,\eta)-x\del_\eta m(x,\xi,\eta)]\d x.
\]
Assuming that $m$ solves the differential equation \eqref{eq:boman_cond}, we introduce the differential operators:
\[
D_a:=\del_\eta+a(\xi,\eta),\quad D_b:=\del_\xi-b(\xi,\eta),
\]
where $a$ and $b$ are strictly positive smooth functions. To see that this makes sense, observe that 
\begin{multline*}
D_b R_m[f](\xi,\eta)-D_a R_m[xf](\xi,\eta)=\\
\int_\R f(x,\xi x+\eta)[\del_\xi m(x,\xi,\eta)-x\del_\eta m(x,\xi,\eta)-(b(\xi,\eta)+xa(\xi,\eta))m(x,\xi,\eta)]\d x.
\end{multline*}
This is zero if the $m$ satisfies the partial differential equation \eqref{eq:boman_cond}:
\begin{equation*}
\del_\xi m(x,\xi,\eta)-x\del_\eta m(x,\xi,\eta)=(xa(\xi,\eta)+b(\xi,\eta))m(x,\xi,\eta).
\end{equation*}
As mentioned in the introduction, this condition on $m$ first appeared in \cite{Gindikin} and is discussed also in \cite{Bo1}. Assuming that it holds, we have derived the relation
\[
D_b R_m[f](\xi,\eta)=D_a R_m[xf](\xi,\eta).
\]
However, $D_a$ and $D_b$ do not commute in general so the analysis will be quite different in the weighted case.

Using standard methods from the theory of differential equations we may derive expressions for the inverse operators of $D_a$ and $D_b$:
\[
D_a^{-1}=\e^{-A}(H\circledast_\eta\e^A\cdot),\quad D_b^{-1}=\e^{B}(H\circledast_\xi\e^{-B}\cdot),
\]
where $H$ denotes the Heaviside function and 
\[
A=A(\xi,\eta)=\int_{-\infty}^\eta a(\xi,\eta')\d\eta',\quad B(\xi,\eta)=\int_{-\infty}^\xi b(\xi',\eta)\d\xi'.
\]
For our purposes it is enough to conclude that for $h\in C^1(\R^2),\ \supp h\subset\{(\xi,\eta);\xi\geq\xi_0\in\R,\eta\geq\eta_0\in\R\}$, it holds that $D_a^{-1}D_ah=h$ and $D_b^{-1}D_bh=h$. One can easily verify the following result:
\begin{proposition}
\label{prop:iter_inv}
\[
(D_a^{-1})^k=\e^{-A}(H_k\circledast_\eta\e^A\cdot),\quad (D_b^{-1})^k=\e^{B}(H_k\circledast_\xi\e^{-B}\cdot),
\]
where $H_{k+1}(x)=\frac{x_+^k}{k!}$.
\end{proposition}
\begin{remark}\label{rmk:comm}
Denote the usual commutator bracket by $[D_a,D_b]=D_aD_b-D_bD_a$, then for any $g\in C^2(\R^2)$, 
$D_bD_ag=D_aD_bg-[D_a,D_b]g$. It is easy to see that $[D_a,D_b]$ is an order zero differential operator since
\begin{multline*}
[D_a,D_b]g=(\del_\eta+a)(\del_\xi g-bg)-(\del_\xi-b)(\del_\eta g+ag)\\
=\del_\eta\del_\xi g-\del_\eta(bg)+a\del_\xi g-abg-(\del_\xi\del_\eta g+\del_\xi(ag)-b\del_\eta g-bag)\\
=-g\del_\eta b-b\del_\eta g+a\del_\xi g-g\del_\xi a-a\del_\xi g+b\del_\eta g=-g(\del_\eta b+\del_\xi a).
\end{multline*}
Furthermore, $[D_a,D_b]=0$ when $\del_\eta b=-\del_\xi a$. This is a Cauchy-Riemann type of equation which is fulfilled if for example $b$ and $a$ are the real- and imaginary parts of a holomorphic function in $\zeta=\xi+\i\eta$, or in other words $a$ is the harmonic conjugate of $b$. If $[D_a,D_b]=0$ on some open set $U\supset\supp R_m[f]$ and $m$ satisfies \eqref{eq:boman_cond} then we do get an identity similar to \eqref{id:1},
\begin{equation*}
D_a^kR_m[x^kf](\xi,\eta)=D_b^kR_m[f](\xi,\eta).\label{eq:comm}
\end{equation*}
As this is a very special case we will not consider it in more detail but simply observe that several arguments that we have to go through for more general $m$ could be simplified.
\end{remark}

\subsection{Moment estimates II}
\label{sec:Wmeans}
In the case when we assume only that $a=a(\xi,\eta)$ and $b=b(\xi,\eta)$ are smooth over some set $U\supset\supp R_m[f]$, $[D_a,D_b]$ is in general non-zero. We will first add the assumption that the functions $a,b\in C^{\omega}(\R^2)$, i.e.\ that $a$ and $b$ are real analytic. Suppose also that $\supp a\subset\{(\xi,\eta);\eta>-\gamma\}$ and recall that
\[
A(\xi,\eta)=\int_{-\gamma}^\eta a(\xi,\eta')\d\eta', \quad\eta<\gamma.
\]
Suppose that $\varphi_N\in C_0^\infty([-1,1])$ is a function as in Proposition \ref{LHseq} that is in addition even with $\int\varphi_N=1$, $\varphi_{N,\eps}(\xi)=\eps^{-1}\varphi_N(\eps^{-1}\xi)$ and fix $\eta=\gamma>0$. By a change of variables ($y=\xi x+\gamma,x\neq0$) we have the following identity between means of the weighted Radon transform of moments of $f$ (over an $\eps$-wide cone of lines through $\eta=\gamma$) and moments of the means $\mathcal{M}_{\eps,\gamma}[fm_\gamma]$ (where $m_\gamma$ is defined in \eqref{m_gamma} after Definition \ref{d:hor}):
\begin{multline*}
\int_{-\eps}^\eps R_m[x^kf](\xi,\gamma)\varphi_{N,\eps}(\xi)\d\xi=\int_{-\eps}^\eps \int_\R x^kf(x,\xi x+\gamma)m(x,\xi,\gamma)\d x\,\varphi_{N,\eps}(\xi)\d\xi\\
=\int_{|x|\leq x_{\eps,\gamma}} x^k\mathcal{M}_{\eps,\gamma}[fm_\gamma](x)\d x.
\end{multline*}

We define $D_a$ and $D_b$ as before and assuming that condition \eqref{eq:boman_cond} holds, we also had the identity
\begin{equation}\label{eq:id1}
D_aR_m[xf](\xi,\eta)=D_bR_m[f](\xi,\eta).
\end{equation}
If we introduce $g_k(\xi,\eta)=R_m[x^kf](\xi,\eta)$, abbreviating the $k$:th Radon moment we also get from \eqref{eq:id1},
\begin{equation}\label{eq:id2}
D_ag_k(\xi,\eta)=D_bg_{k-1}(\xi,\eta).
\end{equation}
Consequently, $g_k=D_a^{-1}D_bg_{k-1}$, where (for $h=h(\xi,\eta)$)
\[
D_a^{-1}h(\xi,\eta)=\e^{-A(\xi,\eta)}(H\circledast_\eta\e^{A}h)(\xi,\eta)=\e^{-A(\xi,\eta)}\int_{-\infty}^\eta\e^{A(\xi,\eta')}h(\xi,\eta')\d\eta'.
\]
In the above $\circledast_\eta$ is convolution in the $\eta$-variable and $A(\xi,\eta)=\int_{-\infty}^\eta a(\xi,\eta')\d\eta'$. Introducing $\psi_A(\xi,\eta',\eta)=\e^{A(\xi,\eta')-A(\xi,\eta)}$, we have shown
\begin{equation}\label{eq:id3}
g_k(\xi,\eta)=\int_{-\infty}^\eta\psi_A(\xi,\eta,\eta')D_bg_{k-1}(\xi,\eta')\d\eta'.
\end{equation}
Iterating we get $g_k(\xi,\eta)=(D_a^{-1}D_b)^kg_0(\xi,\eta)$, where $g_0(\xi,\eta)=R_m[f](\xi,\eta)$. We will return to this identity after presenting some necessary simple lemmas on operators with structure similar to $(D_a^{-1}D_b)^k$. Our main goal of this section will be to prove Proposition \ref{p:p2}.

Denote by $\mathcal P$ the set of integral operators $P$ in the $\eta$-variable with $\xi$ as a parameter of the form 
\begin{equation*}
Pu(\xi, \eta) = \int_{-\infty}^{\eta} p(\xi, \eta, \eta') u(\xi, \eta') \d\eta' , 
\end{equation*}
where $p(\xi, \eta, \eta')$ is a smooth function of all variables. The elements of $\mathcal P$ will be considered as linear operators on the set of smooth functions $u(\xi, \eta)$ that are supported in some halfspace $\eta \ge c$. Assume for simplicity that $c=0$ from here on, but later $-\gamma$ will take the role of $c$.

It is clear that $\mathcal P$ is a ring under composition. For each integer $k\ge 0$ we define the subring $\mathcal P_k$ that is generated by all products of $k$ factors 
\begin{equation*}
P_1 P_2 \ldots P_k
\end{equation*} 
with each $P_j \in \mathcal P$.  We shall denote by $\del_{\xi}P$ the operator with Schwartz kernel 
$\del_{\xi} p(\xi, \eta, \eta')$. 
It is clear that $\mathcal P_k$ is a two-sided ideal in $\mathcal P$ and that  
$P \in \mathcal P_k$ implies 
$\del_{\xi} P \in \mathcal P_k$. We note also that  $\mathcal P_1 = \mathcal P$. 
We shall also consider the operator $\del_{\xi}: \, u \mapsto \del_{\xi}u$.  Note that 
\begin{equation}       \label{pxiR}
 \del_{\xi} \circ R  =  R \, \del_{\xi}   +  \del_{\xi}R
\end{equation}
for all operators $R \in \mathcal P$. 

The operator $D_a^{-1} D_b$ can be written $P \del_{\xi} + Q$ for some $P$
 and $Q$ in $\mathcal P$. Using \eqref{pxiR} we can alternatively write 
 $D_a^{-1} D_b = \del_{\xi} \circ P + Q_1$ with $Q_1 = Q - \del_{\xi} P$. 

\begin{lemma}
The operator $P$ is in $\mathcal P_k$ if an only if its kernel can be factored 
\begin{equation*}
p(\xi, \eta, \eta') = (\eta - \eta')^{k-1} p_0(\xi, \eta, \eta')
\end{equation*}
for some smooth function $p_0(\xi, \eta, \eta')$. 
\end{lemma}

\begin{proof}   
Since $\xi$ plays no role in the argument we shall forget about it. 
Assume that 
$p(\eta, \eta') = (\eta - \eta')  q(\eta, \eta')$. Then 
\begin{multline*}
\del_\eta Pu(\eta) = \int_0^\eta q(\eta,\eta')u(\eta')\d\eta' + \int_0^\eta (\eta-\eta') q'_\eta(\eta,\eta') u(\eta') \d\eta'\\ 
=  \int_0^\eta w(\eta,\eta')u(\eta')\d\eta' = Wu(\eta) , 
\end{multline*}
where 
\begin{equation*}
w(\eta,\eta') = q(\eta,\eta') + (\eta-\eta') q'_\eta(\eta,\eta'), \quad 0 < \eta' < \eta .  
\end{equation*}
It follows that $P = HW$, where $H$ is the integration operator $Hu(t) = \int_0^t u(s) ds$, 
hence $P \in \mathcal P_2$. Repeated use of this argument proves that $P \in \mathcal P_k$ if $p(\eta,\eta')$ is divisible by $(\eta-\eta')^{k-1}$.  Conversely, let $P$ and $Q$ be operators with kernels $p(\eta,\eta')$ and $q(\eta,\eta')$, respectively. Then, by a simple change of variables, the kernel of $R = P Q$ can be seen to be equal to 
\begin{equation}    \label{r1}
r(\eta,\eta') = \int_{\eta'}^\eta p(\eta, v) q(v, \eta')\d v . 
\end{equation}
Another change of variable $v = \eta' + v_1(\eta-\eta')$ gives 
\begin{equation}     \label{r2}
r(\eta,\eta') = (\eta-\eta')\int_0^1 p(\eta, \eta' + v_1(\eta-\eta')) q(\eta' + v_1(\eta-\eta'), \eta') \d v_1 ,  
\end{equation}
which proves the statement for the case $k=2$. (The "only if" part is trivial if $k=1$.)  An obvious induction argument proves the general case.
\end{proof}

\begin{lemma}\label{l:l2}
Let $P$, $Q$, and $R = PQ$ be as above, and assume that 
\begin{equation*}
p(\eta,\eta') = (\eta-\eta')^{k-1} p_0(\eta,\eta')  , 
\end{equation*}
where $|p_0(\eta,\eta')| \le M_1$ and that $|q(\eta,\eta')| \le M_2$.  
Then 
\begin{equation*}
r(\eta,\eta') = (\eta-\eta')^{k} r_0(\eta,\eta') ,  
\end{equation*}
where $|r_0(\eta,\eta')| \le M_1 M_2/k$.  
\end{lemma}

\begin{proof}
Arguing as in  \eqref{r2} we obtain $r(\eta,\eta') = (\eta-\eta')^{k} r_0(\eta,\eta')$,  where
\begin{equation}     \label{r0}
r_0(\eta,\eta') =   \int_0^1 (1 - v_1)^{k-1} \,  p_0(\eta, \eta' + v_1(\eta-\eta')) q(\eta' + v_1(\eta-\eta'), \eta') \d v_1 . 
\end{equation}
By the assumption it follows that 
\begin{equation}     \label{lem2} 
|r_0(\eta,\eta')| \le   M_1 M_2   \int_0^1 (1 - v_1)^{k-1}  dv_1 
  \le  M_1 M_2 /k , 
\end{equation}
which completes the proof.
\end{proof}



\begin{proposition}\label{p:p2}
The operator 
$(D_a^{-1} D_b)^k = (\del_{\xi} \circ P + Q)^k$  can be written 
\begin{equation}       \label{eq11}
(D_a^{-1} D_b)^k =   \sum_{j=0}^k \del_{\xi}^j \circ S_{j,k} 
\end{equation}
where $S_{j,k} \in \mathcal P_k$. Assume that $C\geq1$ is such that the derivatives of the Schwartz kernels of $P$ and $Q$ are bounded by,  
\begin{equation}       \label{eq1}
|\del_{\xi}^{n} p(\xi, \eta, \eta')| \le C^{n+1}n!, \ \textrm{ and } \  
|\del_{\xi}^{n} q(\xi, \eta, \eta')|   \le C^{n+1}n!, \quad \textrm{$n\in\N$}   .
\end{equation}
Then the Schwartz kernels of $S_{j,k}$ can be estimated with
\begin{equation}       \label{prop3}
|s_{j,k}(\xi, \eta, \eta')| \le (\beta C)^{2k-j} (k - j)! \frac{(\eta - \eta')^{k-1}}{(k-1)!} .   
\end{equation}
where $\beta\geq1+\sqrt3$.
\end{proposition}

We postpone the proof of the above proposition to the end of this section. An important consequence of Proposition \ref{p:p2} is that we are now able to prove:

\begin{proposition} \label{p:p2mom}
The moments of  $\mathcal M_{\eps, \gamma}[fm_\gamma]$ satisfy the estimates 
\begin{equation}       \label{cor}
\left| \int x^k \mathcal M_{\eps, \gamma}[fm_\gamma](x) dx \right| \le \left(\frac{C_1}{\eps}\right)^{k+1}\e^N\|R_m[f]\|_{L^1(R_{\eps,\gamma})}   
\end{equation}
for some constant $C_1>0$ and all $1\leq k\leq N$ for some $N\in\N$.
\end{proposition}

\begin{proof}  
\begin{equation*}
\int x^k \mathcal M_{\eps, \gamma}[f](x) \d x = \int (D_a^{-1} D_b)^k g_0(\xi, \gamma) \varphi_{N,\eps}(\xi) \d\xi ,  
\end{equation*} 
where $g_0 = R_m[f]$. 
(The first factor in the  integrand should of course be interpreted 
$((D_a^{-1} D_b)^k g_0)(\xi, \gamma)$).
Using Proposition \ref{p:p2} and integration by parts we obtain 
\begin{equation*}
\int x^k \mathcal M_{\eps, \gamma}[f](x)\d x 
= \sum_{j=0}^k (-1)^j \int (S_{j,k} g)(\xi, \gamma) \del_{\xi}^j \varphi_{N,\eps}(\xi) \d\xi . 
\end{equation*}
Thus
\begin{multline*}
\left|\int x^k \mathcal M_{\eps, \gamma}[f](x)\d x \right|\leq\sum_{j=0}^k\|s_{j,k}\del_\xi^j\varphi_{N,\eps}\|_\infty\int_{-\gamma}^\gamma\int_{-\eps}^\eps|g_0(\xi,\eta)|\d\xi\d\eta\\
=\|g_0\|_{L^1(R_{\eps,\gamma})}\sum_{j=0}^k\|s_{j,k}\|_\infty\|\del_\xi^j\varphi_{N,\eps}\|_\infty.
\end{multline*}
Now we use that
\begin{equation*}
\sup_{|\eta|<\gamma}|s_{j,k}(\xi, \gamma,\eta)| \le \frac{(\beta C)^{2k-j} (k-j)! (2\gamma)^{k-1}}{(k-1)!} . 
\end{equation*}
Combining this with the estimate 
\begin{equation*}
\|\del_{\xi}^j \varphi_{N,\eps}(\xi)\|_\infty\le \left(\frac{C_2}{\eps}\right)^{j+1}N^j , \quad j\le N ,
\end{equation*}
we obtain 
\begin{equation*}
\|s_{j,k}\|_\infty\|\del_\xi^j\varphi_{N,\eps}\|_\infty\leq\frac{(\beta C)^{2k-j}(k-j)!(2\gamma)^{k-1}C_2^{j+1}N^j}{(k-1)!\eps^{j+1}}.
\end{equation*}
Summing over $j=1,2,\dots,k$ we finally get
\begin{multline*}
\left|\int x^k \mathcal M_{\eps, \gamma}[f](x)\d x \right|\leq\sum_{j=0}^k\frac{(\beta C)^{2k-j}(k-j)!(2\gamma)^{k-1}C_2^{j+1}N^j}{(k-1)!\eps^{j+1}}\|g_0\|_{L^1(R_{\eps,\gamma})}\\
\leq\frac{(\beta C)^{2k}C_2^{k+1}(2\gamma)^{k-1}k}{\eps^{k+1}}\|g_0\|_{L^1(R_{\eps,\gamma})}\sum_{j=0}^k\frac{(k-j)!j!}{k!}\frac{N^j}{j!}\\\leq\left(\frac{C_1}{\eps}\right)^{k+1}\e^N\|g_0\|_{L^1(R_{\eps,\gamma})}.
\end{multline*}
\end{proof}
Again, note that for $k=0$ \eqref{cor} does not make sense. So we introduce some new norm, similarly as for the standard Radon transform in \eqref{newnorm}. Then we have shown an analogue of Proposition \ref{p:mom} also for $m$ satisfying \eqref{eq:boman_cond} with $a,b\in C^\omega(\R^2)$. Thus we can deduce that Theorem \ref{main2} (and corresponding Theorems \ref{main2var} and \ref{f_stability_var} for $\alpha>1/2$) also hold for $f$ and $R$ replaced with $fm_\gamma$ and $R_m$ respectively.

In order to prove Proposition \ref{p:p2} we first present some auxiliary lemmas.
\begin{lemma}\label{l:l4}
Let $u $ and $v$ be functions (of one variable) satisfying the estimates 
\begin{equation*}
|\del^n u| \le A_1 B_1^n n! , \qquad  |\del^n v| \le A_2 B_2^n n! \quad \textrm{for all $n\in\N$} .  
\end{equation*}
Then with $A = \max(A_1,A_2), B = \max(B_1, B_2)$
\begin{equation*}
|\del^n(uv)| \le A^2 B^n (n+1)!  . 
\end{equation*}
\end{lemma}

\begin{proof}
By Leibnitz' formula 
\begin{align}   \label{lem4} 
|\del^n(uv)| & \le \sum_{p+q=n} \frac{n!}{p! q!} A_1 B^p p! A_2 B^q q! 
= n! A_1 A_2 \sum_{p+q=n} B^{p+q}   \notag \\ 
& \leq  n! A^2 (n+1) B^n = A^2B^n (n+1)!  . 
\end{align}
\end{proof}

Combining Lemma \ref{l:l2} and Lemma \ref{l:l4} we immediately obtain the following estimate.

\begin{lemma}\label{l:l5}
Let $P\in \mathcal P_k$, $Q \in \mathcal P$, $R = PQ$, and assume that 
\begin{equation*}
|\del_{\xi}^n p(\xi, \eta, \eta')|  \le   AB^n n!  \frac{(\eta - \eta')^{k-1}}{(k-1)!}, \quad 
|\del_{\xi}^n q(\xi, \eta, \eta')|  \le AB^n n! \quad \textrm{for $n\in\N$. }
\end{equation*}
Then 
\begin{equation*}
|\del_{\xi}^n r(\xi, \eta, \eta')| \le A^2B^n (n+1)! \frac{(\eta - \eta')^{k}}{k!}    
\quad \textrm{for all $n\in\N$} . 
\end{equation*}
\end{lemma}

\begin{proof}
Writing $p(\xi, \eta, \eta') = p_0(\xi, \eta, \eta') (\eta - \eta')^{k-1}/(k-1)!$ we have 
\begin{equation*}
r(\xi, \eta, \eta') =  \frac{(\eta - \eta')^{k}}{(k-1)!} \int_0^1 (1 - v)^{k-1} \,  
p_0(\xi, \eta, \eta' + v(\eta - \eta')) q(\xi, \eta' + v(\eta - \eta'), \eta') \d v . 
\end{equation*}
Operating with $\del_{\xi}^n$ under the integral sign and using Lemma \ref{l:l4} we obtain
\begin{multline*}
|\del_{\xi}^n r| \le  \frac{(\eta - \eta')^{k}}{(k-1)!} 
      \int_0^1 (1 - v)^{k-1} \,  A^2B^n (n+1)! \d v \\
  \le A^2 B^n (n+1)! \frac{(\eta - \eta')^{k}}{k!}  , 
\end{multline*}
which proves the assertion. 
\end{proof}

\begin{proof}[{\it Proof of Proposition \ref{p:p2}.}] 
The Schwartz kernel of $S_{j,k}\in\mathcal{P}_k$ can according to the arguments in the proof of Lemma \ref{l:l5} be written
\begin{equation}
s_{j,k}(\xi,\eta,\eta')\frac{(\eta-\eta')^{k-1}}{(k-1)!}.\label{eq2}
\end{equation}
The function $s_{j,k}$ satisfies
\begin{equation}
|\del_\xi^ns_{j,k}|\leq (\beta C)^{n+2k-j}(k-j+n)!\ ,\quad k,n\in\N, \beta\geq1+\sqrt3,C\geq1. \label{eq3}
\end{equation}
We prove \eqref{eq3} by induction. Observe that for $k=1$ we have $s_{0,1}=q$ and $s_{1,1}=p$ so \eqref{eq1} implies that \eqref{eq3} holds with $\beta=1$ in this case. Assuming that \eqref{eq11} and \eqref{eq3} hold for $k$, then
\begin{multline}
(D_a^{-1}D_b)^{k+1}=\left(\sum_{j=0}^{k}\del_\xi^j\circ S_{j,k}\right)(\del_\xi\circ P+Q)=\sum_{j=0}^{k}\del_\xi^j\circ S_{j,k}(\del_\xi\circ P+Q)\\
=\sum_{j=0}^{k}\del_\xi^j\circ(S_{j,k}\del_\xi\circ P+S_{j,k}Q)=\sum_{j=0}^{k}\del_\xi^j\circ((\del_\xi\circ S_{j,k}-\del_\xi S_{j,k})P+S_{j,k}Q)\\
=S_{0,k}Q-\del_\xi S_{0,k}P+\sum_{j=1}^k \del_{\xi}^j\circ(S_{j,k}Q-\del_\xi S_{j,k}P+S_{j-1,k}P)+\del_\xi^{k+1}\circ S_{k,k}P.\label{eq4}
\end{multline}
From the above we identify
\begin{eqnarray}
S_{0,k+1}&=&S_{0,k}Q-\del_\xi S_{0,k}P,\label{eq21}\\
S_{j,k+1}&=&S_{j,k}Q-\del_\xi S_{j,k}P+S_{j-1,k}P,\quad1\leq j\leq k,\label{eq22}\\
S_{k+1,k+1}&=&S_{k,k}P.\label{eq23}
\end{eqnarray}
We will finish the proof by deriving estimates for $|\del_\xi^n(s_{j,k}p)|$ (identical arguments will work for $p$ replaced by $q$) and $|\del_\xi^n(\del_\xi s_{j,k}p)|$ for $0\leq j\leq k$. First, by \eqref{eq1} and \eqref{eq3}:
\begin{multline}
|\del_\xi^n(s_{j,k}p)|\leq\sum_{i=0}^n\binom{n}{i}|\del_\xi^{n-i}s_{j,k}||\del_\xi^ip|\\
\leq\sum_{i=0}^n\binom{n}{i}(\beta C)^{n-i+2k-j}(k-j+n-i)!C^{i+1}i!\\
=(\beta C)^{n+2k-j+1}(k-j+n)!\sum_{i=0}^n\frac{\binom{n}{i}}{\binom{n+k-j}{i}}\beta^{-i-1}\\
\leq(\beta C)^{n+2k-j+1}(k-j+n)!\sum_{i=0}^m\beta^{-i-1},\label{eq31}
\end{multline}
and
\begin{multline}
|\del_\xi^n(\del_\xi s_{j,k}p)|\leq\sum_{i=0}^n\binom{n}{i}(\beta C)^{n+1-i+2k-j}(k-j+n+1-i)!C^{i+1}i!\\
=(\beta C)^{n+2k-j+2}(k+1-j+n)!\sum_{i=0}^n\frac{\binom{n}{i}}{\binom{n+k+1-j}{i}}\beta^{-i-1}\\
\leq(\beta C)^{n+2k-j+2}(k+1-j+n)!\sum_{i=0}^m\beta^{-i-1}.\label{eq32}
\end{multline}
Using the two above estimates we can move on to
\begin{multline}
|\del_\xi^ns_{j,k+1}|\leq|\del_\xi^n(s_{j,k}q)|+|\del_\xi^n(\del_\xi s_{j,k}p)|+|\del_\xi^n(s_{j-1,k}p)|\\
\leq(\beta C)^{n+2k+2-j}(k+1-j+n)!\left(\sum_{i=0}^n\beta^{-i-1}\right)\left(\frac{1}{\beta C(k+1-j+n)}+1+\frac{1}{\beta C}\right).
\end{multline}
Now
\[
\sum_{i=0}^n\beta^{-i-1}\leq\frac1\beta\sum_{i=0}^\infty\frac1{\beta^i}=\frac1{\beta-1},
\]
and
\[
\frac{1}{\beta C(k+1-j+n)}+1+\frac{1}{\beta C}\leq1+\frac2\beta.
\]
We choose $\beta>1$ such that
\[
\frac1{\beta-1}\left(1+\frac2\beta\right)\leq1.
\]
The above inequality holds if
\[
\beta+2\leq\beta(\beta-1),
\]
or equivalently $\beta^2-2\beta-2=(\beta-1)^2-3\geq0$ which hold if $\beta\geq1+\sqrt3$. Hence, we have shown
\begin{equation}
|\del_\xi^ns_{j,k+1}|\leq(\beta C)^{n+2(k+1)-j}(k+1-j+n)!\ ,\quad 1\leq j\leq k.\label{eq41}
\end{equation}
In a very similar way we get for $j=0$
\begin{equation}
|\del_\xi^ns_{0,k+1}|\leq(\beta C)^{n+2(k+1)}(k+1+n)!\ ,\quad \beta\geq1+\sqrt2,\label{eq42}
\end{equation}
and for $j=k+1$
\begin{equation}
|\del_\xi^ns_{k+1,k+1}|\leq(\beta C)^{n+k+1}n!\ ,\quad \beta\geq\frac{1+\sqrt5}2.\label{eq43}
\end{equation}
This proves \eqref{eq3}, and taking $n=0$ we obtain \eqref{prop3}, and the proof of Proposition \ref{p:p2} is complete.
\end{proof}

\subsection{The case $a,b\in G^\sigma(\R^2)$}
\label{s:gevwts}
Next we sketch how one can get stability estimates also for more general weights $m$ that satisfy \eqref{eq:boman_cond} with $a,b$ belonging to Gevrey spaces, $G^\sigma(\R^2)$,  c.f.\ Definition \ref{def:Gevrey}. Observe that by Proposition \ref{prop:Gev2}, for every $\eta$, $\exp[A(\xi,\eta)-A(\xi,\gamma)]\in G^\sigma(\R)$, when considered as a function of $\xi$. Recall that
\[
A(\xi,\eta)=\int_{-\gamma}^\eta a(\xi,\eta')\d\eta', \quad\eta<\gamma.
\]
We can then, by doing only minor adjustments, prove the following version of Proposition \ref{p:mom}:

\begin{proposition}\label{prop:qest}
If $m$ satisfies equation \eqref{eq:boman_cond} with $a,b\in G^\sigma(\R^2)$ and $\supp f\subset\{(x,y);y\geq x^2\}$
\[
|m_k(\mathcal{M}_{\eps,\gamma}[fm_\gamma])|\leq \left(\frac{C}{\eps}\right)^{k+1}k!^s\|R_m[f]\|_{\eps,\gamma},\quad k=0,1,\dots
\]
where $s=\sigma-1$ and $C>0$.
\end{proposition}

The most notable difference in the arguments is that $\mathcal{M}_{\eps,\gamma}=\mathcal{M}_{\phi,\eps,\gamma}$ can now be defined with a test function $\phi\in G_0^\sigma(\R)$, instead of using a sequence. The proof is then more or less identical with the proof of Proposition \ref{p:p2mom}, but derivatives need to be estimated using $|\del^mp|\leq C^{m+1}m!^\sigma$ etc.\ for some $\sigma>1$. 

Next we will use the Legendre polynomials over $|x|\leq1$ together with Parseval's identity, as in section \ref{s:rt_est}, to finish the proof of a stability estimate in the case of the weighted Radon transform. In the proof we put emphasis on the few details that need to be changed.

\begin{theorem}\label{mainest_wt}
If in addition to the assumptions in Proposition \ref{prop:qest} we assume that $0<\eps<1$, $\gamma>0$ small enough and $\|f\|_{C^{0,\alpha}(\R^2)}\leq C_0$, then
\[
\|\mathcal{M}_{\eps,\gamma}[fm_\gamma]\|_2\leq4M\left(\frac{\log(C/\eps)\log\log(M/\|R_m[f]\|_{\eps,\gamma})}{\log(M/\|R_m[f]\|_{\eps,\gamma})}\right)^\alpha,
\]
for small $\|R_m[f]\|_{\eps,\gamma}$. $M>0$ depends on $C_0$ and $C$ depends on $C_0$ and $s=\sigma-1$.
\end{theorem}

\begin{proof}
Proposition \ref{prop:qest} together with Lemma \ref{l:legmom} implies with $H=\|R_m[f]\|_{\eps,\gamma}$,
\[
|a_n|\leq\left(\frac{C}{\eps}\right)^{k+1}n!^sH
\]
which in turn gives
\[
\left(\sum_{n=0}^{N}|a_n|^2\right)^{1/2}\leq\left(\frac{C}{\eps}\right)^{N+1}N^{sN}H\leq\frac{M}{N^\alpha}.
\]
The last inequality is equivalent with
\[
N\log N\left(s+\frac{\alpha}{N}+\left(1+\frac1N\right)\frac{\log(C/\eps)}{\log N}\right)\leq\log\frac{M}{H}
\]
which holds if
\[
N\log N\leq\frac{\log(M/H)}{\log(C(s)/\eps)}=y,
\]
where $C(s)$ depends only on $s$. In turn the above inequality holds if we choose $N$ as the largest integer such that
\[
N\leq\frac{y}{\log y},
\]
then
\[
N\geq\frac{y}{\log y}-1.
\]
Choosing $N$ in this way and taking Lemma \ref{l:tailest} into account we have proven
\[
\|\mathcal{M}_{\eps,\gamma}[fm_\gamma]\|_2\leq\frac{2M}{N^\alpha}\leq2M\left(\frac{\log y}{y-\log y}\right)^\alpha\leq4M\left(\frac{\log y}{y}\right)^\alpha.
\]
Hence
\begin{multline*}
\|\mathcal{M}_{\eps,\gamma}[fm_\gamma]\|_2\leq4M\left(\frac{\log(C(s)/\eps)\log\left(\frac{\log(M/H)}{\log(C(s)/\eps)}\right)}{\log(M/H)}\right)^\alpha\\
\leq4M\left(\frac{\log(C(s)/\eps)\log\log(M/H)}{\log(M/H)}\right)^\alpha.
\end{multline*}
\end{proof}
Using similar arguments as in the proof of Theorem \ref{f_stability} we can also prove the next theorem.
\begin{theorem}\label{fm_stability}
Under the assumptions in Theorem \ref{mainest_wt} it holds that
\[
\|f(\cdot,\gamma)m_\gamma(\cdot,0,\gamma)\|_{2}\leq M\frac{\log^{2\alpha}\log(\|R_m[f]\|_{\eps,\gamma}^{-1})}{\log^\alpha(\|R_m[f]\|_{\eps,\gamma}^{-1})}
\]
for small $\|R_m[f]\|_{\eps,\gamma}$ and $M$ depends on $C_0$ and $s=\sigma-1$.
\end{theorem}
\begin{proof}
Choosing $\eps<\eps_0$ implies $H_\eps=\|R_m[f]\|_{\eps,\gamma}\leq\|R_m[f]\|_{\eps_0,\gamma}=H$ so if 
\[
\frac1{\log(M/H)}\leq\eps\leq\frac2{\log(M/H)},
\]
we get
\begin{multline*}
\|f(\cdot,\gamma)\|_2\leq4M\left(\frac{(\log C(s)+\log(1/\eps))\log\log(M/H_\eps)}{\log(M/H_\eps)}\right)^\alpha+M\eps^\alpha\\
\leq4M\left(\frac{(\log C(s)+\log\log(M/H))\log\log(M/H)}{\log(M/H)}\right)^\alpha+2M\frac1{\log^\alpha(M/H)}\\
\lesssim M\left(\frac{\log^2\log(\|R_m[f]\|_{\eps_0,\gamma}^{-1})}{\log(\|R_m[f]\|_{\eps_0,\gamma}^{-1})}\right)^\alpha
\end{multline*}
\end{proof}
We can also get a version of Theorem \ref{f_stability_var}, using similar arguments, for $a,b\in G^\sigma(\R^2)$.
\begin{remark}
In the special case when $a$ and $b$ are as in Remark \ref{rmk:comm}, i.e.\ if $D_a$ and $D_b$ commute, it is very easy to prove a similar result as Proposition \ref{prop:qest} where only the Propositions \ref{prop:Gev2} and \ref{prop:Gev1} are needed.
\end{remark}

%% file: Appendix.tex
\label{sec:app}

\subsection{A special sequence of test functions}
\label{s:horseq}
In order to improve our continuity estimate for analytic weights we need to choose a test function out of a particular sequence. The simple construction of this sequence can be found in H\"ormander's book \cite{LH1}, or in Rodino's book \cite{Rodino} (Proposition 1.4.10).
\begin{proposition}\label{LHseq}
For any given neighborhood $V$ of $x_0\in\R^n$ we can find a sequence $\varphi_N\in C_0^\infty(V)$ such that $0\leq\varphi_N$ and
\begin{equation}\label{LHseqest}
|\del^\alpha\varphi_N|\leq C^{|\alpha|+1}N^{|\alpha|},\quad|\alpha|\leq N,
\end{equation}
where $C>0$ is a constant not depending on $N$ or $\alpha$.
\end{proposition}

\subsection{The Gevrey classes $G^{\sigma}$}
\label{sec:Gevrey}
In this section we recall some basic properties of the classes $G^\sigma$. Most of these results can be found in Rodino's book, \cite{Rodino} on Gevrey spaces.
\begin{definition}
\label{def:Gevrey}
Let $\Omega\subset\R^n$ and $\sigma\geq1$. If $f\in C^\infty(\Omega)$ and for every compact subset $K$ of $\Omega$ there exists a $C>0$ such that for all multi-indices $\alpha$ and all $\bx\in K$
\begin{equation}
|\del^\alpha f(\bx)|\leq C^{|\alpha|+1}(\alpha!)^\sigma,\label{eq:Gev1}
\end{equation}
then we call $f$ a \emph{Gevrey function of order $\sigma$}. We also write $f\in G^\sigma(\Omega)$.
\end{definition}
\begin{remark}
In place of estimates \eqref{eq:Gev1} one sometimes use the equivalent
\begin{equation}
|\del^\alpha f(\bx)|\leq C^{|\alpha|+1}|\alpha|^{\sigma|\alpha|}.\label{eq:Gev2}
\end{equation}
Also observe that we may assume $f\geq0$ which is illustrated by the following example of a $G_0^\sigma(\R)$-function from \cite{Ramis}:
\[
\phi(x)=\begin{cases}
\exp\left(-\frac{1}{((1-t)t)^{\frac1{\sigma-1}}}\right)&,t\in(0,1)\\
0&,t\notin(0,1).
\end{cases}
\]
\end{remark}
\begin{remark}
Observe that $G^1(\Omega)$ is the set of all analytic functions on $\Omega$.
\end{remark}
\begin{definition}
For $\sigma>1$ we define $G_0^\sigma(\Omega)$ to be the set of all $f\in G^\sigma(\Omega)$ such that $f$ has compact support.
\end{definition}
The proof of the following propositions can be found in Rodino's book.
\begin{proposition}
\label{prop:Gev1}
$G^\sigma(\Omega)$ is a vector space and a ring (with respect to multiplication of functions) and is closed under differentiation. 
\end{proposition}
In order to treat compositions we first need the following definition.
\begin{definition}
\label{def:Gmap}
Suppose that $\chi:\Omega\to\Lambda$, where $\Lambda$ is an open subset of $\R^m$. We say that $\chi$ is an \emph{$G^\sigma$-map} if $\chi=(\chi_1,\dots,\chi_m)$ and $\chi_j\in G^\sigma(\Omega)$ for all $j$.
\end{definition}
\begin{proposition}
\label{prop:Gev2}
Suppose that $\chi:\Omega\to\Lambda$ is a $G^\sigma$-map and $f\in G^\sigma(\Lambda)$, then $f\circ\chi\in G^\sigma(\Omega)$.
\end{proposition}
The proposition is proved for analytic maps $\chi$ in \cite{Rodino}, but the more general proposition is proved in \cite{LionsMag}.

\subsection{Legendre polynomials}
\label{sec:LP}
Recall that the Legendre polynomials $\{P_n\}_{n=0}^\infty$ forms a complete orthogonal system over $L^2([-1,1])$. There are many equivalent definitions, for example:
\begin{equation}
P_n(x)=\frac1{2^n}\sum_{k=0}^{[n/2]}(-1)^k\binom{n}{k}\binom{2n-2k}{n}x^{n-2k},\label{eq:Legendre}
\end{equation}
where $[\cdot]$ denotes the round to closest integer function. Some properties of $P_n$ include:
\begin{itemize}
\item For even/odd $n, P_n(x)$ is an even/odd polynomial in $x$,
\item $\deg P_n=n$,
\item $|P_n(x)|\leq1$,
\item $\|P_n\|_2^2=\frac{2}{2n+1}$, where $\|P\|_2^2=\langle P,P\rangle=\int_{-1}^1 P(x)\overline{P(x)}\d x$.
\end{itemize}
Given any $f\in L^2([-1,1])$ we know that
\[
f(x)=\sum_{n=0}^\infty\frac{\langle f,P_n\rangle}{\|P_n\|_2^2} P_n(x),
\]
where the convergence is in $L^2([-1,1])$-sense. Denote from now on the (Fourier-) Legendre coefficients by
\[
a_n=\frac{\langle f,P_n\rangle}{\|P_n\|_2}=\sqrt{\frac{2n+1}{2}}\int_{-1}^1 f(x)P_n(x)\d x.
\]
Parseval's identity reads in this case,
\[
\|f\|_2^2=\sum_{n=0}^\infty|a_n|^2.
\]
The following lemmas were used in the proofs of the stability estimates.
\begin{lemma}\label{l:legmom}
\[
|a_n|\leq(4\sqrt2)^n\max_{0\leq k\leq n}|m_k(f)|.
\]
where
\[
m_k(f)=\int_{-1}^1x^kf(x)\d x,\quad k=0,1,\dots,n
\]
are moments $f$.
\end{lemma}

\begin{proof}
\begin{multline*}
|a_n|=\left|\int_{-1}^1 f(x)\sqrt{\frac{2n+1}{2}}\frac1{2^n}\sum_{k=0}^{[n/2]}\binom{n}{k}\binom{2n-2k}{n}x^{n-2k}\d x\right|\\
\leq\sqrt{\frac{2n+1}{2}}\frac1{2^n}\sum_{k=0}^{[n/2]}\binom{n}{k}\binom{2n-2k}{n}\left|\int_{-1}^1x^{n-2k}f(x)\d x\right|\\
\leq2^{n/2}\cdot2^n\max_{0\leq k\leq n}|m_k(f)|\sum_{k=0}^{[n/2]}\binom{n}{k}\\
\leq(4\sqrt2)^n\max_{0\leq k\leq n}|m_k(f)|.
\end{multline*}
The first inequality is just the triangle inequality, the second inequality follows from that $\binom{2n-2k}{n}\leq\binom{2n}{n}\leq4^n$ and $\sqrt{\frac{2n+1}{2}}<2^{n/2}$. The final inequality follows by increasing the number of terms (from $[n/2]$ to $n$) in the sum and applying the binomial theorem.
\end{proof}
There are many results on convergence and the magnitude of $a_n$ depending on the regularity of $f$. Good references are Sansone's and Alexits' books \cite{Sansone,Alexits}. Recall that the modulus of continuity of a function $f$ on $\R^n$ is defined by the quantity
\[
\omega(f;r)=\sup_{\|\bx-\by\|\leq r}|f(\bx)-f(\by)|.
\]
In particular, $f\in C^{0,\alpha}(\R)$ implies that there are constant $0<\alpha\leq1,C_0>0$ such that
\[
\omega(f;r)\leq C_0 r^\alpha.
\]
As the best $L^2$-approximation of a function $f$ on $[-1,1]$ in terms of an $N$:th degree polynomial is given by the Fourier-Legendre sum:
\[
S_N[f](x)=\sum_{n=0}^Na_n\tilde P_n(x), \textrm{where } \tilde P_n(x)=\frac{P_n(x)}{\|P_n\|_2},
\]
we can easily verify the following well-known lemma
\begin{lemma}\label{l:tailest}
Suppose that $\|f\|_{C^{0,\alpha}[-1,1]}\leq C_0$, then
\[
\|f-S_N[f]\|_2\leq\sqrt{2}A_0C_0\left(\frac{2}{N}\right)^\alpha,
\]
where $A_0$ is some absolute constant.
\end{lemma}

\begin{proof}
\[
\|f-S_N[f]\|_2^2=\int_{-1}^1|f(x)-S_N[f](x)|^2\d x\leq\int_{-1}^1|f(x)-p_N(x)|^2\d x
\]
for any $N$:th degree polynomial $p_N$ since $S_N[f]$ is the best approximation in $L^2$. Let now in particular $p_N$ be the degree $N$ polynomial that is the best uniform approximation of $f$ on $[-1,1]$. By Jackson's inequality (Theorem 4.6.6 in \cite{Alexits}) it then follows that
\[
\|f(x)-p_N(x)\|_\infty\leq A_0\omega\left(f;\frac{2}{N}\right)\leq A_0C_0\left(\frac{2}{N}\right)^\alpha.
\]
Thus
\[
\|f-S_N[f]\|_2\leq A_0C_0\left(\frac{2}{N}\right)^\alpha\sqrt{\int_{-1}^1\d x}=\sqrt{2}A_0C_0\left(\frac{2}{N}\right)^\alpha.
\]
\end{proof}
In Szegö's book on orthogonal polynomials, \cite{Szego} one can also find better and more explicit bounds on $P_n$ for $n>0$, e.g.
\begin{equation}
(1-x^2)^{1/4}|P_n(x)|\leq\sqrt{\frac{2}{\pi n}}.\label{szegobd}
\end{equation}
Using \eqref{szegobd} together with $\|P_n\|_2=\sqrt{\frac{2}{2n+1}}$ (and $P_0(x)=1$) one can easily verify
\begin{lemma}\label{l:szbd}
If $P_n, n=0,1,\dots$ are Legendre polynomials over the interval $[-1,1]$, then the $L^2$-normalized Legendre polynomials $\tilde{P}_n$ satisfy
\[
|\tilde{P}_n(x)|\leq\begin{cases}2^{1/4}\sqrt{\frac{2n+1}{\pi n}}&,\quad n\geq1\\  \frac1{\sqrt{2}}&,\quad n=0\end{cases}
\]
for $|x|\leq 1/2$.
\end{lemma}